%% file: Strictly_Outer_Actions_Of_Locally_Compact_Groups__Beyond_The_Full_Factor_Case.tex
\documentclass[10pt,twoside,a4paper]{article}

\input{preamble.tex}

\input{macros.tex}

\allowdisplaybreaks
\begin{document}

\title{Strictly outer actions of locally compact groups: beyond the full factor case}
    \author{Basile Morando}
    \date{\today}
    \maketitle
\begin{abstract}
 We show that, given a continuous action $\alpha$ of a locally compact group $G$ on a factor $M$, the relative commutant $M'\cap(M\rtimes_{\alpha} G)$ is contained in $M\rtimes_{\alpha} H$ where $H$ is the subgroup of elements acting without spectral gap. As a corollary, we answer a question of Marrakchi and Vaes by showing that if $M$ is semifinite and $\alpha_g$ is not approximately inner for all $g\neq 1$, then $M'\cap (M\rtimes_{\alpha} G)=\cc$.
\end{abstract}

\section{Introduction and results}

Given a continuous action $\alpha:G\acts M$ of a locally compact group $G$ on a von Neumann algebra $M$, we can build the \emph{crossed product} $M\rtimes_{\alpha} G$: this algebra is generated by a copy of $M$ and a unitary representation of $G$ such that, given $u_{g}$ in the representation of $G$ and $x$ in the copy of $M$, we have $u_gxu_g^*=\alpha_g(x)$. Understanding how properties of the action reflect into properties of the algebra $M\rtimes_{\alpha} G$ and of the inclusion $M\subset M\rtimes_{\alpha} G$ is a fundamental question in the theory of non-commutative dynamical systems. In this article, we focus on the following question: 
\begin{Prob}\label{prob:problem 1}
 Find a criterion on $\alpha:G\acts M$ a continuous action of a locally compact group on a factor $M$, ensuring that $M'\cap(M\rtimes_{\alpha} G)$ is trivial. Such an action is said to be \emph{strictly outer}, and this property trivially implies factoriality of $M\rtimes_{\alpha}G$.
\end{Prob}

When the group $G$ is discrete, many properties of the crossed product can be deduced from the \emph{Fourier decomposition}: every element of $M\rtimes_{\alpha} G$ is uniquely written as a sum $\sum_{\gamma\in G}x_{\gamma}u_{\gamma}$ (where $x_{\gamma}$ and $u_{\gamma}$ respectively refers to elements of the copy of $M$ and $G$ inside $M\rtimes_{\alpha} G$). It allows in particular to answer problem~\ref{prob:problem 1}.
\begin{itemize}
    \item When $M=\cc$, $\cc\rtimes G$ is just the group von Neumann algebra $L(G)$. This algebra is a factor if and only if $G$ has infinite non-trivial conjugacy classes.
    \item When $M$ is a factor, the relative commutant $M'\cap (M\rtimes_{\alpha} G)$ is trivial if and only if $\alpha$ is \emph{outer}: for all $\gamma\neq 1$ in $G$, there is no unitary $u\in M$ such that $\alpha_{\gamma}=\Ad(u)$. 
    \item More generally, consider the subgroup $H=\{\gamma\in G \mid \alpha_{\gamma} \text{ is inner} \}$. Then $M'\cap (M\rtimes_{\alpha}G)\subset M\rtimes_{\alpha}H$. When $M$ is a type $\II_1$ factor, it implies that $M\rtimes_{\alpha} G$ is a factor if and only if the subgroup $H$ of inner-acting elements of $G$ is $G$-icc (for every $h\in H\setminus \{1\}$, the set $\{\gamma h \gamma^{-1}\mid \gamma\in G\}$ is infinite). 
\end{itemize}
In the non-discrete case, these questions become much more subtle: 
\begin{itemize}
    \item Even in the simplest case where $M=\cc$, we do not have any criterion for factoriality of $\cc\rtimes G=L(G)$.
    \item When $M$ is a factor, it is still true that $M'\cap(M\rtimes_{\alpha} G)=\cc$ implies that $\alpha$ is outer. However, the converse is no longer true. For example, if $M$ is a factor of type $\III_0$ with trivial $T$-invariant, any action $\sigma_t^{\phi}:\rr^{+*}\acts M$ by modular automorphisms is outer, but $M\rtimes_{\sigma_t^{\phi}}\rr^{+*}$ is not even a factor. As explained in \cite[Example 8]{marrakchivaesSpectral}, there also exist outer actions $\alpha:\ss^1\acts R$ of $\ss^1$ on the hyperfinite $\II_1$ factor such that $R\rtimes_{\alpha}\ss^1$ is not a factor either.  
    \item Consider $(M,\tau)$ a $\II_{\infty}$ factor, $\tau$ being a semifinite faithful trace on $M$, and $\alpha:\rr^{+*}\acts M$ a continuous action. Assume that $\alpha$ scales the trace, meaning that for all $s\in \rr^{+*}$ and $x\in M$, $\tau(\alpha_s(x))=s^{-1}\tau(x)$. It trivially implies that $\alpha$ is outer, but the fact that $M'\cap (M\rtimes_{\alpha}\rr^{+*})=\cc$  is a deep result of modular theory, the \emph{Connes-Takesaki relative commutant theorem} \cite[Theorem 5.1]{connestakesaki77} (see also \cite[XII.1.7 and XII.1.1]{takesakiTheoryII}).
\end{itemize}

In \cite{marrakchivaesSpectral}, Marrakchi and Vaes were able to completely solve Problem~\ref{prob:problem 1} for \emph{full} factors (in the $\II_1$ setting, factors that do not have property $\Gamma$).
\begin{The}[{\textmd{\cite[Theorem A]{marrakchivaesSpectral}}}]\label{the:marrakchivaes outer implies strictly outer}
 Let $\alpha:G\acts M$ be a continuous action of a locally compact group on a full factor with separable predual. Assume that $\alpha$ is outer. Then $M'\cap(M\rtimes_{\alpha} G)=\cc$.
\end{The}

Recall that a factor $M$ with separable predual is full if and only if the subgroup $\Inn(M)$ of inner automorphisms is closed in the Polish group $\Aut(M)$ \cite{Connes_almost_periodic_1974}. Therefore, any outer automorphism $\theta$ of a full factor has formally stronger properties:
\begin{itemize}
    \item $\theta$ is not in the closure $\overline{\Inn}(M)$ of $\Inn(M)$ in the topological group $\Aut(M)$ (it is not approximately inner). 
    \item The bimodule $\ldeux(\theta)$ does not weakly contain the standard bimodule $\ldeux(\id)$ by \cite[Theorem 3.1]{connes76} and \cite[Proposition 7.2]{marrakchi2020full}.
 If $M$ is a $\II_1$ factor, this is equivalent to the following property: there exists $\epsilon>0$ and $a_1,...,a_n\in M$ such that for all $b\in M$, 
            \[ \sum_i\norm{\theta(a_i)b-ba_i}_2^2\geq \epsilon \norm{b}_2^2.\]
\end{itemize}
In the non-full setting, the properties above define distinct refinements of outerness. We therefore need to introduce corresponding terminologies: for $\theta \in \Aut(M)$ we say that

    \begin{center}
        \renewcommand{\arraystretch}{1.3}
    \begin{tabular}[h!]{|l|c|}
        \hline
        \bf{Terminology}&\bf{Definition}\\\hline\hline
        $\theta$ is inner & $\theta\in \Inn(M)$\\\hline
        $\theta$ is outer & $\theta\notin \Inn(M)$\\\hline
        $\theta$ is topologically (or approximately) inner & $\theta\in \lowoverline[\Inn]{\Inn}(M)$\\\hline
        $\theta$ is topologically outer & $\theta\not\in \lowoverline[\Inn]{\Inn}(M)$\\\hline
        $\theta$ is spectrally (or weakly) inner & $\ldeux(\id)\prec \ldeux(\theta)$\\\hline
        $\theta$ is spectrally outer (or $\theta$ acts with spectral gap) & $\ldeux(\id)\nprec\ldeux(\theta)$ \\\hline
    \end{tabular}
    \end{center}

    We recall the known implications between the various notions of outerness introduced above. The missing implications are known not to hold.
    \begin{center}
        \renewcommand{\arraystretch}{1.3}
    \begin{tabular}[h!]{|l|c|}
        \hline
        $M$ factor              & $\theta$ spectrally outer $\Longrightarrow$ $\theta$ topologically outer $\Longrightarrow$ $\theta$ outer\\
        $M$ $\II_1$ factor      & $\theta$ spectrally outer $\stackrel{\text{\cite{connes76}}}{\Longleftrightarrow}$ $\theta$ topologically outer $\Longrightarrow$ $\theta$ outer \\
        $M$ full factor         & $\theta$ spectrally outer $\stackrel{\text{\cite{connes76,marrakchi2020full}}}{\Longleftrightarrow}$ $\theta$ topologically outer $\Longleftrightarrow$ $\theta$ outer\\
        \hline
    \end{tabular}
    \end{center}

 When $\alpha:G\acts M$ is a continuous action on a factor, we say that $\alpha$ is inner (resp. outer, topologically inner, topologically outer, spectrally inner, spectrally outer) if for all $g\in G\setminus\{1\}$, $\alpha_g$ is inner (resp. outer, topologically inner, topologically outer, spectrally inner, spectrally outer). 
    
 As pointed out before, Theorem~\ref{the:marrakchivaes outer implies strictly outer} does not hold for non-full factors. Nevertheless, Marrakchi and Vaes ask in \cite{marrakchivaesSpectral} for the following natural generalization: 

    \begin{Ques}[{\textmd{\cite[Question 9]{marrakchivaesSpectral}}}]\label{ques:topologically outer implies strictly outer}
 Let $\alpha:G\acts M$ be a continuous action of a locally compact group on a factor. Assume that for all $g\in G\setminus\{1\}$, $\alpha_g\notin \overline{\Inn}(M)$ (i.e. $\alpha$ is topologically outer). Do we have $M'\cap(M\rtimes_{\alpha} G)=\cc$?
    \end{Ques}

In this paper, we solve this question at least when $M$ is a semifinite factor. This is a consequence of the following main theorem.

\begin{mainthm}\label{thm:location of relative commutant in spectrally inner automorphisms}
 Let $\alpha:G\acts M$ be an action of a locally compact group on a $\sigma$-finite factor. Let $H=\{g\in G\mid \ldeux(\id)\prec \ldeux(\alpha_g) \}< G$ be the subgroup of elements acting without spectral gap. Then  
    \[ M'\cap(M\rtimes_{\alpha} G)\subset M\rtimes_{\alpha} H.\]
    \end{mainthm}

This theorem does not answer Question~\ref{ques:topologically outer implies strictly outer} when $M$ is a type $\III$ factor because not every topologically outer action on a type $\III$ factor is spectrally outer. However, this holds true when $M$ is of type $\II_1$, which implies the following corollary:

\begin{maincor}\label{thm:top outer action are str outer}
 Let $\alpha:G\acts M$ be a continuous action of a locally compact group on a $\sigma$-finite type $\II$ factor. If $\alpha$ is topologically outer (i.e. for all $g\in G\setminus \{1\}$, $\alpha_g\notin \overline{\Inn}(M)$), then $M'\cap (M\rtimes_{\alpha} G)=\cc$.
\end{maincor}

If $\alpha$ is a trace scaling action on a $\II_{\infty}$ factor, it is topologically outer. Therefore, Corollary~\ref{thm:top outer action are str outer} generalizes Connes-Takesaki's relative commutant theorem. However, our proof of this corollary uses Connes-Takesaki's result, and therefore we do not get a new proof of this deep result. 

To establish Theorem~\ref{thm:location of relative commutant in spectrally inner automorphisms} the key step is the following \emph{uniform spectral gap for group actions} result: 

\begin{mainthm}\label{thm:fundamental result on continuity of weak containment}
 Let $\alpha:G\acts M$ be a continuous action of a locally compact group on a $\sigma$-finite factor.  
 Let $g\in G$ such that $\ldeux(\id)\nprec \ldeux(\alpha_g)$. Then there exists an open neighborhood $\O$ of $g$ in $G$ such that $\ldeux(\id)\nprec \bigoplus_{h\in \O}\ldeux(\alpha_h)$. 
\end{mainthm}

Specializing this result in the $\II_1$ factors case makes it clearer: saying that $g$ is such that $\ldeux(\id)\nprec\ldeux(\alpha_g)$ is exactly saying that there exists $\epsilon>0$ and $a_1,...,a_n\in M$ such that
\[\forall \xi \in \ldeux M, \sum_{i}\norm{\alpha_g(a_i)\xi-\xi a_i}_2^2\geq \epsilon \norm{\xi}_2^2.\] 
By Connes' result, this is equivalent to $\alpha_g\notin\overline{\Inn}(M)$. Therefore, by continuity of the action, there is an open neighborhood $\O$ of $g$ in $G$ such that for all $h\in \O$, $\alpha_h$ possesses a similar spectral gap property, but a priori with $\epsilon$ and $a_1,...,a_n$ depending on the group element. Theorem~\ref{thm:fundamental result on continuity of weak containment} exactly says that it is possible to choose a neighborhood $\O$ such that the spectral gap is uniform over $\O$:  there exists $\epsilon>0$, $a_1,...,a_n\in M$ such that for all $h$ in $\O$,
\[ \forall \xi\in \ldeux M, \sum_{i}\norm{\alpha_h(a_i)\xi-\xi a_i }_2^2\geq \epsilon \norm{\xi}_2^2.\]

Whereas in the full case, this result (see \cite[lemma 1]{marrakchivaesSpectral}) is a direct consequence of a uniform spectral gap result \emph{at the automorphism group level} by Jones \cite{jones_centralsequences_1982}, it is not clear at all that such a property (or a satisfying weakening of it) remains true for non-full factors. We obtain Theorem~\ref{thm:fundamental result on continuity of weak containment} using different methods: we rely on the locally compact group structure of G to apply smoothing procedures to well-chosen states.

\section{Preliminaries}

\paragraph*{Topologies on von Neumann algebras.}
Let $M$ be a von Neumann algebra, $M_*$ its predual. By weak topology on $M$ we mean the weak-$*$ topology relative to $M_*$. The strong topology (resp. the $*$-strong topology) is the topology induced by the seminorms $\norm{x}_{\phi}\doteq \phi(x^*x)^{1/2}$ (resp. $\norm{x}_{\phi}^{\sharp}\doteq \phi(x^*x)^{1/2}+\phi(xx^*)^{1/2}$) for $\phi\in M_*^+$. 

\paragraph*{Automorphisms group.}
The \emph{automorphism group} $\Aut(M)$ acts on $M_*$ by $\theta(\phi)=\phi\circ\theta$ for $\theta\in\Aut(M)$ and $\phi\in M_*$. Following \cite{Haagerup_1975}, it is endowed with the $u$-topology which is the topology of pointwise norm convergence on $M_*$: a net $(\theta_i)_{i\in I}$ converges to $\theta$ if and only if for all $\phi\in M_*$, $\norm{\theta_i(\phi)-\theta(\phi)}\to 0$. With this topology $\Aut(M)$ is a topological group that is Polish when $M_*$ is separable. The map 
\begin{align*}
 \Ad:\U(M)&\to \Aut(M)\\
 u&\mapsto [x\mapsto uxu^*]
\end{align*}
is continuous, and its image is the subgroup of \emph{inner} automorphisms denoted by $\Inn(M)$. When $M$ is a factor, we say that it is \emph{full} if the map $\Ad$ is open on its range \cite{Connes_almost_periodic_1974}. When $M_*$ is separable, $M$ is full if and only if $\Inn(M)=\overline{\Inn}(M)$, where $\overline{\Inn}(M)$ is the closure of $\Inn(M)$ in the $u$-topology. The topological group of classes of outer automorphisms is defined by $\Out(M)=\Aut(M)/\Inn(M)$, and the quotient map is denoted by $p:\Aut(M)\to \Out(M)$. When $\Inn(M)$ is not closed, we will consider $\underline{\Out}(M)=\Aut(M)/\overline{\Inn}(M)$ and the corresponding quotient map $\underline{p}:\Aut(M)\to\underline{\Out}(M)$. 

\paragraph*{Crossed products.} We refer to \cite[Chapter X,\S 1]{takesakiTheoryII} for proofs and details. 
Let $G$ be a locally compact group. We denote by $\mu$ one left Haar measure on it, by $\lambda:G\to \ldeux (G,\mu)$ (resp. $\rho:G\to \ldeux (G,\mu)$) its left (resp. right) regular representation and by $L(G)$ (resp. $R(G)$) the von Neumann algebra $\lambda(G)''$ (resp. $\rho(G)''$).  A \emph{continuous action} $\alpha:G\acts M$ of $G$ on a von Neumann algebra $M$ is a continuous group homomorphism $\alpha:G\to\Aut(M)$. 
The \emph{crossed product} associated to the action $\alpha:G\acts M$ is the von Neumann algebra generated by a copy of $M$ and a family of unitaries $\{u_g\mid g\in G\}$ verifying $u_gxu_g^*=\alpha_g(x)$ for $g\in G$ and $x\in M$ that admits a representation $\alpha:M\rtimes_{\alpha} G\to M\ptvn B(\ldeux G)$ such that $\alpha(u_g)=1\otimes \lambda_g$ and $\alpha(x)=[s\mapsto \alpha_{s^{-1}}(x)]\in M\ptvn \linf(G)\subset M\ptvn B(\ldeux G)$. 

\paragraph*{Smoothing procedure for von Neumann algebra valued functions.}

Let $\alpha:G\acts M$ be a continuous action on a von Neumann algebra. Given $x\in M$, the map $g\mapsto \alpha_g(x)$ is weakly continuous, but generally not norm continuous. Consider $M_{\alpha}=\{x\in M\mid g\mapsto\alpha_g(x) \text{ is norm continuous}\}$: $M_{\alpha}$ is an $\alpha$ invariant $C^*$-subalgebra of $M$, and it turns out that $M_{\alpha}$ is $*$-strongly dense in $M$. This is a consequence of the following useful smoothing procedure.  

First, recall that we can integrate some von Neumann algebra valued functions, even if they are not Bochner measurable. Let $f:(X,\mu)\to M$ be a map from a measure space to a von Neumann algebra with $x\mapsto \norm{f(x)}$ integrable. It is \emph{weakly measurable} if for all $\omega\in M_*$, $\omega\circ f:X\to\cc$ is measurable. For such a function, there exists a unique element $\int_Xfd\mu\in M$ such that for all $\omega\in M_*$, $\omega(\int_Xfd\mu)=\int_X(\omega\circ f)d\mu$. For all $x\in M$, $h\mapsto \alpha_h(x)$ is weakly integrable and therefore, for any Borel $X\subset G$ of non-zero finite Haar measure, we can define the \emph{smoothing map} $\alpha_{X}$ by 
\begin{align*}
    \alpha_{X}: M & \rightarrow M\\
 x & \mapsto \frac{1}{\mu(X)}\int_{X}\alpha_h(x)d\mu(h).
\end{align*}
 For all $ X\subset G$, $0<\mu(X)<\infty$ and for all $x\in M$, $\alpha_X(x)\in M_{\alpha}$. Furthermore, if $\B$ is a neighborhood basis of the identity in $G$, $\lim_{U\in \B}\alpha_{U}(x)=x$ in weak topology which implies the $*$-strong density of $M_{\alpha}$ in $M$. Note that for any $x\in M_{\alpha}$, $h\mapsto \alpha_h(x)$ is norm-continuous. Therefore, it is Bochner measurable which implies that for all $\omega \in M^*$ and $x\in M_{\alpha}$, $\omega(\alpha_X(x))=\frac{1}{\mu(X)}\int_{X}\omega(\alpha_h(x))d\mu(h)$ (this does not hold for a generic element in $M$). 

\paragraph*{Comparisons of bimodules.} We refer to \cite{connes1994noncommutative} and \cite{anantharaman1995amenable} for proofs and details. Given $M$ and $N$ two von Neumann algebras, an $M$-$N$-bimodule $(\H,\pi_{\H})$ (sometimes abbreviated to $\pi_{\H}$ or to $\H$) is a binormal $*$-representation of the algebra $M\pta N^{\op}$. We denote by $\lambda_{\H}$ and $\rho_{\H}$ the restriction of $\pi_{\H}$ to $M$ and $N^{\op}$, and for $\xi\in \H$, $a\in M$ and $b\in N$ we use the abbreviations: $a\cdot \xi\doteq \lambda_{\H}(a)\xi$ and $\xi\cdot b\doteq \rho_{\H}(b)\xi$ when there is no risk of confusion. 
There are several ways of comparing two $M$-$N$-bimodules $(\H,\pi_{\H})$ and $(\K,\pi_{\K})$:
\begin{itemize}
\item  $\MNbimod{\H}$ is \emph{contained} in $\MNbimod{\K}$ if there exists an isometry $V:\H\to\K$ intertwining the representations $\pi_{\H}$ and $\pi_{\K}$. If we can choose $V$ as a unitary, we say that the two bimodules are \emph{unitarily equivalent}, and we write $(\H,\pi_{\H})\simeq(\H,\pi_{\K})$. This is in fact equivalent to $\MNbimod{\H}\subset\MNbimod{\K}$ and $\MNbimod{\K}\subset \MNbimod{\H}$. We denote by $\Hom((\H,\pi_{\H}),(\K,\pi_{\K}))$ the set of linear maps between $\H$ and $\K$ that intertwine the two representations, and we put $\End(\H,\pi_{\H})\doteq\Hom((\H,\pi_{\H}),(\H,\pi_{\H}))$: it is exactly the von Neumann algebra $\pi_{\H}(M\pta N^{\op})'\subset B(\H)$. 

\item $\MNbimod{\H}$ is \emph{quasi-contained} in $\MNbimod{\K}$ if there exists a normal $*$-homomorphism $\Phi:\pi_{\K}(M\pta N^{\op})''\to \pi_{\H}(M\pta N^{\op})''$ such that the following diagram commutes.

\[\begin{tikzcd}
 {M\pta N^{\op} } & {\pi_{\K}(M\pta N^{\op})''} \\
 & {\pi_{\H}(M\pta N^{\op})''}
    \arrow["{\pi_{\K}}", from=1-1, to=1-2]
    \arrow["{\pi_{\H}}"'{pos=0.4}, from=1-1, to=2-2]
    \arrow["\Phi", from=1-2, to=2-2]
\end{tikzcd}\]

Equivalently, $\MNbimod{\H}$ is quasi-contained in $\MNbimod{\K}$ if there exists a central projection $z\in\Z(\pi_{\K}(M\pta N^{\op})'')$ such that $\pi_{\H}(M\pta N^{\op})''$ is isomorphic to $z\left[\pi_{\K}(M\pta N^{\op})''\right]$. If $\MNbimod{\H}$ is quasi-contained in $\MNbimod{\K}$ and $\MNbimod{\K}$ is quasi-contained in $\MNbimod{\H}$, we say that $\MNbimod{\H}$ and $\MNbimod{\K}$ are \emph{quasi-equivalent}. When $p$ is a projection in $\pi_{\H}(M\pta N^{\op})'$, $p\H$ is naturally endowed with an $M$-$N$-bimodule structure. If $p$ and $q$ are two such projections, $\MNbimod{p\H}$ is quasi-contained in $ \MNbimod{q\H}$ if and only if $z(p)\leq z(q)$, $z(p)$ being the central support of $p$ (check \cite[Chapter 6]{bekka2019unitary} for details).

\item $\MNbimod{\H}$ is \emph{weakly contained} in $\MNbimod{\K}$, written $\MNbimod{\H}\prec\MNbimod{\K}$ if there exists a $*$-homomorphism $\Phi:\overline{\pi_{\K}(M\pta N^{\op})}^{\norm{\cdot}}\to\overline{\pi_{\H}(M\pta N^{\op})}^{\norm{\cdot}}$ such that the following diagram commutes. % https://q.uiver.app/#q=WzAsMyxbMCwwLCJNXFxvdGltZXNfe2FsZ31OXntvcH0gIl0sWzEsMCwiQ14qX3tcXHBpX3tIfX0oTVxcb3RpbWVzX3thbGd9Tl57b3B9ICkiXSxbMSwxLCJDXipfe1xccGlfe0t9fShNXFxvdGltZXNfe2FsZ31OXntvcH0gKSJdLFswLDIsIlxccGlfe0t9IiwyLHsibGFiZWxfcG9zaXRpb24iOjQwfV0sWzAsMSwiXFxwaV97SH0iXSxbMSwyLCJcXFBoaSJdXQ==
\[\begin{tikzcd}
 {M\pta N^{\op} } & \overline{\pi_{\K}(M\pta N^{\op})}^{\norm{\cdot}} \\
 & \overline{\pi_{\H}(M\pta N^{\op})}^{\norm{\cdot}}
    \arrow["{\pi_{\K}}", from=1-1, to=1-2]
    \arrow["{\pi_{\H}}"'{pos=0.4}, from=1-1, to=2-2]
    \arrow["\Phi", from=1-2, to=2-2]
\end{tikzcd}\]

Equivalently, $\MNbimod{\H}\prec\MNbimod{\K}$ if for all $x\in M \pta N^{\op}$ (equivalently, for all $x\in M\otimes_{\max}N^{\op}$), $\norm{\pi_{\H}(x)}\leq \norm{\pi_{\K}(x)}$. If $\MNbimod{\H}\prec\MNbimod{\K}$ and $\MNbimod{\K}\prec\MNbimod{\H}$, $\MNbimod{\H}$ and $\MNbimod{\K}$ are \emph{weakly equivalent}, and we write $\MNbimod{\H}\sim\MNbimod{\K}$.
\end{itemize}

As indicated by the terminology, containment of bimodules implies quasi-containment which in turn implies weak containment. We will need the following formulation of weak containment in terms of states. 
\begin{Prop}\label{prop:equivalent_caracterisation_weak_containment}
 Let $\MNbimod{\H}$ and $\MNbimod{\K}$ be $M$-$N$-bimodules. Suppose that $\H$ admits a cyclic vector for $\pi_{\H}$. The following properties are equivalent:
    \begin{enumerate}
        \item $\MNbimod{\H}\prec \MNbimod{\K}$\label{itm:ecwc 1}
        \item For all unit vector $\xi\in\H$ there exists a state $\omega\in B(\K)^*$ such that $\omega(\pi_{\K}(a\otimes b^{\op}))=\ps{\pi_{\H}(a\otimes b^{\op})\xi}{\xi}$ for all $a\in M$, $b\in N$.\label{itm:ecwc 2}
        \item There exists a $\pi_{\H}$-cyclic unit vector $\xi_0\in\H$ and a state $\omega\in B(\K)^*$ such that  $\omega(\pi_{\K}(a\otimes b^{\op}))=\ps{\pi_{\H}(a\otimes b^{\op})\xi_0}{\xi_0}$ for all $a\in M$, $b\in N$.\label{itm:ecwc 3}
    \end{enumerate}
\end{Prop}
\begin{proof}\; 
    \begin{enumerate}[itemindent=26pt]
        \item[$\ref{itm:ecwc 1} \Rightarrow~\ref{itm:ecwc 2}$] Let $\Phi:\overline{\pi_{\K}(M\pta N^{\op})}^{\norm{\cdot}}\to\overline{\pi_{\H}(M\pta N^{\op})}^{\norm{\cdot}}$ be the $*$-homomorphism given by the weak containment, and $\xi\in \H$ a unit vector. Extending $x\mapsto \ps{\Phi(x)\xi}{\xi}$ to $B(\K)$ by Hahn-Banach produces a state with the required properties. 
        \item[$\ref{itm:ecwc 3} \Rightarrow~\ref{itm:ecwc 1}$] Let $x\in M\otimes_{\max} N^{\op}$.
            \begin{flalign*}
 \norm{\pi_{\H}(x)} &=\norm{\pi_{\H}(\abs{x})}
                        = \sup_{\norm{\xi}=1}\ps{\pi_{\H}(\abs{x})\xi}{\xi}
                        =\sup_{\substack{y\in M\otimes_{\max} N^{\op}\\ \norm{\pi_{\H}(y)\xi_0}\leq 1}}\ps{\pi_{\H}(\abs{x})\pi_{\H}(y)\xi_0}{\pi_{\H}(y)\xi_0} & & \\
                    &= \sup_{\substack{y\in M\otimes_{\max} N^{\op}\\ \norm{\pi_{\H}(y)\xi_0}\leq 1}}\ps{\pi_{\H}(y^*\abs{x}y)\xi_0}{\xi_0}
                        = \sup_{\substack{y\in M\otimes_{\max} N^{\op}\\ \norm{\pi_{\H}(y)\xi_0}\leq 1}} \omega\circ \pi_{\K}(y^*\abs{x}y)& &\\
                    %Takesaki I.9.10
                    &\leq \norm{\pi_{\K}(\abs{x})}\sup_{\substack{y\in M\otimes_{\max} N^{\op}\\\norm{\pi_{\H}(y)\xi_0}\leq 1}} \omega\circ \pi_{\K}(y^*y)
                        = \norm{\pi_{\K}(\abs{x})}\sup_{\substack{y\in M\otimes_{\max} N^{\op}\\\norm{\pi_{\H}(y)\xi_0}\leq 1}} \ps{\pi_{\H}(y^*y)\xi_0}{\xi_0} & &\\
                    &\leq \norm{\pi_{\K}(\abs{x})}=\norm{\pi_{\K}(x)}& &
            \end{flalign*} 
    \end{enumerate}
\end{proof}

\paragraph*{Operations on bimodules.} 

Given $(\H,\pi_{H})$ and $M$-$N$-bimodule and $(\K,\pi_{\K})$ an $N$-$P$-bimodule, we denote by $(\H\otimes_{N}\K,\pi_{\H}\otimes_N \pi_{\K})$ the \emph{Connes tensor product} of these two bimodules, which is an $M$-$P$-bimodule \parencite[Appendix B.$\delta$]{connes1994noncommutative}. 

Given $(\H,\pi_{\H})$ an $M$-$N$-bimodule, we denote by $(\overline{H}, \overline{\pi_{\H}})$ the \emph{contragredient} bimodule associated to $(\H,\pi_{\H})$, which is an $N$-$M$-bimodule. It is defined in the following way: if $\overline{\H}$ is the conjugate Hilbert space of $\H$, we denote by $\xi \mapsto \overline{\xi}$ the canonical antilinear isometry from $\H$ to $\overline{\H}$. Then the $N$-$M$ bimodule structure on $\overline{\H}$ is given by $b\cdot \overline{\xi} \cdot a\doteq \overline{a^*\cdot \xi\cdot b^*}$\parencite[Appendix B.$\delta$]{connes1994noncommutative}.

These two operations behaves well with respect to weak containment: 

\begin{Prop}[{\textmd{\cite[Proposition 2.2.1]{popa1986correspondences}}}]\label{Prop:continuity of operations on bimodules}
    Let $\H$ and $\K$ be two $M$-$N$-bimodules such that $\H\prec\K$. Then: 
    \begin{enumerate}
        \item $\overline{\H}\prec\overline{\K}$
        \item  For $\L$ any $P$-$M$-bimodule, $\L\otimes_M\H\prec\L\otimes_M\K$
        \item  For $\L$ any $N$-$P$-bimodule, $\H\otimes_N\L\prec\K\otimes_N\L$
    \end{enumerate}
\end{Prop}

\paragraph*{Classical bimodules.} In this paragraph we refer to \cite{Haagerup_1975} and \cite[Chapter 5 Appendix B]{connes1994noncommutative} for details.
Let $M$ be a von Neumann algebra. A fundamental example of $M$-$M$-bimodule is given by the standard form $(M,\ldeux M, J,\ldeux M^+)$ of $M$, where $\ldeux M$ is a Hilbert space, $J$ the canonical antilinear involution, and $\ldeux M^+$ the positive cone of $\ldeux M$. The Hilbert space $\ldeux M$ is canonically endowed with a bimodule structure: for $\xi \in \ldeux M$ and $a\in M$, we denote by $a\xi$ (resp. $\xi a$) the corresponding left (resp. right) multiplication.  This defines the \emph{standard bimodule}. The standard $M$-$M$-bimodule is a neutral element for the Connes tensor product: $\ldeux(M)\otimes_M\H$, $\H\otimes_M\ldeux(M)$ and $\H$ are unitarily equivalent.

Each automorphism $\theta\in\Aut(M)$ defines an $M$-$M$-bimodule $\ldeux(\theta)\doteq(\ldeux M,\pi_{\theta})$ through $\pi_{\theta}(a\otimes b^{\op})\xi\doteq \theta^{-1}(a)\xi b$ for all $a,b\in M, \xi\in\ldeux M$. In particular, $\ldeux(\id)$ is the standard bimodule.

Given $(X,\mu)$ a measure space, a Borel map $\chi:(X,\mu)\to \Aut(M)$ defines an $M$-$M$-bimodule $\ldeux(\chi)\doteq(\ldeux(X,\mu)\otimes \ldeux M, \pi_{\chi})$ through $\pi_{\chi}(a\otimes b^{\op})\xi(x)\doteq \chi(x)^{-1}(a)\xi(x) b$ for all $a,b\in M$, $x\in X$, $\xi\in \ldeux(X,\mu)\otimes \ldeux M$. In particular, if \(\alpha:G\acts M\) is a continuous action of a locally compact group $G$ on $M$, each Borel subset $A\subset G$ defines a bimodule $\ldeux(\alpha_{\mid A})$.

Here are some standard properties of these bimodules:

\begin{Prop}[{\textmd{\cite[Appendix B.$\delta$]{connes1994noncommutative}}}]\label{Prop:basic properties of bimodules associated to automorphisms}\;
    \begin{enumerate}
        \item Let $\theta\in \Aut(M)$. Then $\overline{\ldeux(\theta)}\simeq\ldeux(\theta^{-1})$. 
        \item Let $\theta_1, \theta_2\in \Aut(M)$. Then $\ldeux(\theta_1)\otimes_M\ldeux(\theta_2)\simeq \ldeux(\theta_1\circ \theta_2)$.
        \item Let $\alpha:G\acts M$ be a continuous action of $G$ on $M$, $A$ be a Borel subset of $G$ and $h\in G$. Then $\ldeux(\alpha_h)\otimes_M\ldeux(\alpha_{\mid A})\simeq\ldeux(\alpha_{\mid h\cdot A})$
    \end{enumerate}
\end{Prop}

\begin{Prop}\label{prop:weak equivalence with essential image}
    Let $(X,\nu)$ be a $\sigma$-finite measure space and $\chi:(X,\nu)\to \Aut(M)$ a Borel map. Then \[\ldeux(\chi)\sim \bigoplus_{\theta\in \essim \chi}\ldeux(\theta),\]
    where $\essim(\chi)$ is the \emph{essential range} of $\chi$ (i.e. the support of the pushforward measure $\chi_{*}\nu$).
\end{Prop}
\begin{proof}
    Let $a,b\in M$. Then \[\norm{\pi_{\chi}(a\otimes b^{\op})}=\esssup_{x\in(X,\nu)}\norm{\pi_{\chi(x)}(a\otimes b^{\op})}=\sup_{\theta\in\essim \chi}\norm{\pi_{\theta}(a\otimes b^{\op})}.\] The first equality is a general result about decomposable operators (see \cite[Section 1.H]{bekka2019unitary}). As for the second, the inequality $\esssup_{x\in(X,\nu)}\norm{\pi_{\chi(x)}(a\otimes b^{\op})}\leq\sup_{\theta\in\essim \chi}\norm{\pi_{\theta}(a\otimes b^{\op})}$ is a simple consequence of the definitions. For the other direction, let $\theta\in \essim(\chi)$. Let $(\U_i)_{i\in I}$ be a basis of neighborhoods of $\theta$ in $\Aut(M)$. For all $i\in I$, $\nu(\chi^{-1}(\U_i))>0$, and by $\sigma$-finiteness of $\nu$ we can take Borel sets $\V_i\subset \chi^{-1}(\U_i)$ with $0<\nu(\V_i)<\infty$. Define $\eta_i=\frac{1}{\nu(\V_i)^{1/2}}1_{\V_i}\in \ldeux(X,\nu)$. For any $\xi\in\ldeux M$ and $a,b\in M$, 
    \[\ps{\pi_{\chi}(a\otimes b^{\op})\xi\otimes \eta_i}{\xi\otimes \eta_i}\tend{i}{\infty}\ps{\pi_{\theta}(a\otimes b^{\op})\xi}{\xi}.\]
    But for all $i\in I$, $\norm{\pi_{\chi}(a\otimes b^{\op})}\norm{\xi}^2=\norm{\pi_{\chi}(a\otimes b^{\op})}\norm{\xi\otimes \eta_i}^2\geq \ps{\pi_{\chi}(a\otimes b^{\op})\xi\otimes \eta_i}{\xi\otimes \eta_i}$. So $\norm{\pi_{\theta}(a\otimes b^{\op})}\leq \norm{\pi_{\chi}(a\otimes b^{\op})}$, and more generally $\sup_{\theta\in \essim(\chi)}\norm{\pi_{\theta}(a\otimes b^{\op})}\leq \norm{\pi_{\chi}(a\otimes b^{\op})}$.
\end{proof}

As a consequence, if $\alpha:G\acts M$ is a continuous action, $A\subset G$ is a Borel subset and $g$ is in the support of $\mu_{\mid A}$, then $\ldeux(\alpha_g)\prec \ldeux(\alpha_{\mid A})$. Note that if $\theta\in\Aut(M)$ is such that $\ldeux(\id)\prec \ldeux(\theta)$, then by propositions \ref{Prop:continuity of operations on bimodules} and \ref{Prop:basic properties of bimodules associated to automorphisms}, $\ldeux(\id)\simeq \overline{\ldeux(\id)}\prec \overline{\ldeux(\theta)}\simeq \ldeux(\theta^{-1})$, and therefore $\ldeux(\id)\sim \ldeux(\theta)$. In particular, if $\alpha:G\acts M$ is a continuous action, the set $H$ of elements $h\in G$ such that $\ldeux(\id)\prec \ldeux(\alpha_h)$ is a subgroup of $G$. 

\begin{Lem}\label{lem:weak equivalence H(HU) et H(U)} Let $\alpha:G\acts M$ be a continuous action of a locally compact group on a von Neumann algebra. Let $H=\{g\in G\mid \ldeux(\id)\prec \ldeux(\alpha_g) \}=\{g\in G\mid \ldeux(\id)\sim \ldeux(\alpha_g) \}$ and $U$ be an open subset of $G$. Then $\ldeux(\alpha_{\mid UH})$ and $\ldeux(\alpha_{\mid U})$ are weakly equivalent. 
\end{Lem}
\begin{proof} Using proposition \ref{prop:weak equivalence with essential image} and the fact that $\essim(\alpha_{\mid UH})=\alpha(UH)$ because $U$ is open, we obtain: 
\begin{align*}
    \ldeux(\alpha_{\mid UH})
        &\sim\bigoplus_{\theta\in UH}\ldeux(\theta)\sim \bigoplus_{g\in U, h\in H}\ldeux(\alpha_{gh})\sim \bigoplus_{g\in U, h\in H}\ldeux(\alpha_g)\otimes_M\ldeux(\alpha_h)\\
        &\sim \bigoplus_{g\in U, h\in H}\ldeux(\alpha_g)\otimes_M\ldeux(\id)\sim \bigoplus_{g\in U}\ldeux(\alpha_g)\\
        &\sim \ldeux(\alpha_{\mid U}).
\end{align*}
\end{proof}

\paragraph*{Unitary implementation of automorphisms.}
Given $\phi\in M_*^+$, we denote by $\phi^{1/2}$ the unique vector in $\ldeux M^+$ such that for all $x\in M$, $\phi(x)=\ps{x\phi^{1/2}}{\phi^{1/2}}$. When $\theta$ is an automorphism of $M$, there exists a unique unitary $U_{\theta}\in B(\ldeux M)$ such that for all $x\in M$, $\theta(x)=U_{\theta} xU_{\theta}^*$: we denote by $\kappa: \Aut(M)\to \U(B(\ldeux M)), \theta\mapsto U_{\theta}$ this \emph{unitary implementation}. When $\alpha:G\acts M$ is a continuous action on $M$, the unitary implementation extends $\alpha$ in a continuous action $\tilde{\alpha}:G\acts B(\ldeux M)$ through $\tilde{\alpha}_g(x):=\kappa(\alpha_g)x\kappa(\alpha_g)^*$. For $a,b\in M$, $X\subset G$, $\tilde{\alpha}_X(\pi_{\id}(a\otimes b^{\op}))=\frac{1}{\mu(X)}\int_{X}\pi_{\id}(\alpha_h(a)\otimes \alpha_h(b)^{\op})d\mu(h)$ (cf \cite[Thm 3.2]{Haagerup_1975}). In particular, $\tilde{\alpha}_X(\pi_{\id}(a\otimes 1))=\pi_{\id}(\alpha_X(a)\otimes 1)$ and $\tilde{\alpha}_X(\pi_{\id}(1\otimes b^{\op}))=\pi_{\id}(1\otimes \alpha_X(b)^{\op})$.
Note also that for all $a,b\in M_{\alpha}$, $\pi_{\id}(a\otimes b^{\op})\in B(\ldeux M)_{\tilde{\alpha}}$.

\paragraph*{Spectral gaps.}\label{paragraph: spectral gaps}

Let $M$ be a $\II_1$ factor. Given an $M$-$M$-bimodule $\K$, it is particularly interesting to check if $\K$ weakly contains $\ldeux(\id)$. For any $\epsilon>0$ and any finite subset $F$ of $M$, consider the subset of $\Bimod(M,M)$ defined by \[W(\epsilon, F)\doteq\left\{ \text{$M$-$M$-bimodules }(\H,\pi_{\H})\mid \exists \xi\in \H,  \sum_{x\in F}\norm{x\cdot\xi-\xi\cdot x}_2^2<\epsilon \norm{\xi}_2^2\right\} .\] Then (see \cite[Proposition 13.3.7]{anantharaman2017introduction}) $\ldeux(\id)\nprec (\K,\pi_{\K})$ if and only if there exists $\epsilon$ and $F$ such that $(\K,\pi_{\K})\notin W(\epsilon, F)$. This proposition allows us to interpret $\ldeux(\id)\nprec (\K,\pi_{\K})$ as a \emph{spectral gap} property. Applied to bimodules associated to automorphisms, it gives: 
\begin{Prop}\label{prop:spectral gap for Borel family of automorphisms}
 Let $M$ be a $\II_1$ factor, and $A$ be a subset of $\Aut(M)$. 
 Then the following statements are equivalent: 
    \begin{enumerate}
        \item There exists $\epsilon>0$ and $a_1,...,a_n$ in $M$ such that for all $\theta$ in $A$ and for all $\xi$ in $\ldeux M$, \[\sum_{i}\norm{\theta(a_i)\xi-\xi a_i}_2^2\geq \epsilon \norm{\xi}_2^2. \]\label{itm:r1}
        \item $\ldeux(\id)\nprec \bigoplus_{\theta\in A}\ldeux(\theta).$ \label{itm:r2}
    \end{enumerate}
\end{Prop}
In particular:
\begin{itemize}
    \item Applied to $A=\{\theta_0\}$ a single automorphism, item~$\ref{itm:r1}$ is a spectral gap property because it says that $0$ is not in the closure of the spectrum of the operator $\sum_{i=1}^n\abs{\lambda(\theta_0(a_i))-\rho(a_i)}^2$. Then proposition~\ref{prop:spectral gap for Borel family of automorphisms} is the analog of the classical fact that a group representation has spectral gap if and only if it does not weakly contain the trivial representation. 
    \item Using Proposition~\ref{prop:weak equivalence with essential image}, if $(X,\nu)$ is a $\sigma$-finite measure space and $\chi:(X,\nu)\to \Aut(M)$ is a Borel map, then $\ldeux(\id)\nprec \ldeux(\chi)$ if and only if there exists $\epsilon>0$ and $a_1,...,a_n$ in $M$ such that for all $\theta$ in $\essim(\chi)$ and for all $\xi$ in $\ldeux M$, \[\sum_{i}\norm{\theta(a_i)\xi-\xi a_i}_2^2\geq \epsilon \norm{\xi}_2^2.\]
\end{itemize}

Recall that, when $M$ is a $\II_1$ factor, $\theta\notin \overline{\Inn}(M)$ (i.e. $\theta$ is topologically outer) if and only if $\theta$ exhibits the spectral gap properties of Proposition~\ref{prop:spectral gap for Borel family of automorphisms}:

\begin{The}[{\textmd{\cite[Theorem 3.1]{connes76}}}]\label{the:spectral gap caracterization of topological outerness} Let $M$ be a factor, and let $\theta\in \Aut(M)$. Consider the two following statements:
\begin{enumerate}
    \item $\ldeux(\id)\nprec \ldeux(\theta)$ ($\theta$ is spectrally outer).\label{itm:p3}
    \item $\theta\notin \overline{\Inn}(M)$ ($\theta$ is topologically outer).\label{itm:p1}
\end{enumerate}
Then $\ref{itm:p3}\implies\ref{itm:p1}$, and if $\theta$ is a trace-preserving automorphism of a type $\II$ factor, then $\ref{itm:p1}\implies\ref{itm:p3}$.
\end{The}
However there exist topologically outer automorphisms that are not spectrally outer, for instance trace scaling automorphisms of the injective $\II_{\infty}$ factor (see \cite[Introduction]{marrakchi2020full} for examples of such automorphisms on type $\III$ factors).
\begin{proof}
    $\ref{itm:p3}\implies\ref{itm:p1}$. Let $\theta\in\overline{\Inn}(M)$ and choose a net $(u_i)_{i\in I}$ of unitaries such that $\Ad(u_i)\to \theta$. Denote by $\hat{1}$ the standard cyclic unit vector in $\ldeux(M)$. Consider the net of states $(\omega_i)_{i\in I}$ defined by $\omega_i(x)=\ps{x u_i^*\hat{1}}{u_i^*\hat{1}}$ for $x\in \B(\ldeux(M))$. Then for $a,b\in M$, 
    \[\omega_i(\pi_{\theta}(a\otimes b^{\op}))=\ps{\theta^{-1}(a)u_i^*\hat{1}b}{u_i^*\hat{1}}=\ps{u_i\theta^{-1}(a)u_i^*\hat{1}b}{\hat{1}}\tend{i}{\infty}\ps{a\hat{1}b}{\hat{1}}.\]
    Therefore, any weak-$*$ accumulation point $\omega$ of $(\omega_i)_{i\in I}$ will verify $\omega(\pi_{\theta}(a\otimes b^{\op}))=\ps{a\hat{1}b}{\hat{1}}$ for all $a,b\in M$ and therefore, by Proposition \ref{prop:equivalent_caracterisation_weak_containment}, $\ldeux(\id)\prec \ldeux(\theta)$.

    $\ref{itm:p1}\implies\ref{itm:p3}$(for $\II_1$ factors). This is \cite[Theorem 3.1]{connes76}.

    $\ref{itm:p1}\implies\ref{itm:p3}$(for trace preserving actions on semifinite factors). Let $\theta\in \Aut(M)$ a trace-preserving automorphism of the $\II_{\infty}$ factor $(M,\tau)$ such that $\ldeux(\id)\prec \ldeux(\theta)$. Take $p_1\in M$ a projection with $\tau(p_1)=1$:  $\tau(\theta(p_1))=\tau(p_1)$ so there exists a unitary $u_1\in M$ such that $\Ad(u_1)\circ \theta (p_1)=p_1$. Then $\Ad(u_1)\circ\theta$ stabilizes $(1-p_1)M(1-p_1)$, and therefore we can find $p_2<1-p_1$ with $\tau(p_2)=1$ and $u_2$ a unitary in $M$ such that $\Ad(u_2u_1)\circ\theta(p_1)=p_1$ and $\Ad(u_2u_1)\circ\theta(p_2)=p_2$. Inductively, we may consider a partition of the unit $(p_n)_{n\in \nn}$ and a unitary $u\in \U(M)$ such that $\Ad(u)\circ\theta (p_i)=p_i$ for all $i\in \nn$. Put $\beta=\Ad(u)\circ\theta\in \Aut(M)$. Then $\ldeux(\id)\prec \H(\beta_{p_iMp_i})$ for $i$ in $\nn$, and thus, applying~\ref{itm:p1}$\implies$\ref{itm:p3} in the $\II_1$ setting, $\beta_{p_iMp_i}\in \overline{\Inn}(p_iMp_i)$. Therefore, there exists a sequence of unitaries $(v_{n}^{i})_{n\in \nn}$ in $p_iMp_i$ such that $v_n^i\tend{n}{\infty}\beta_{p_i Mp_i}$ for all $i\in\nn$. Taking $v_n=\sum_{i\in \nn} v_n^i$, we get $\beta\in \overline{\Inn}(M)$. But then $\theta=\Ad(u^*)\circ \beta$ is in $\overline{\Inn}(M)$ too. 
\end{proof}

Take $M$ a factor: given an element $\theta\in \Aut(M)\setminus \overline{\Inn}(M)$, $\theta$ admits a neighborhood $\O$ consisting of topologically outer automorphisms. If $M$ is a $\II_1$ factor, we can apply Theorem~\ref{the:spectral gap caracterization of topological outerness} and Proposition~\ref{prop:spectral gap for Borel family of automorphisms} to each $\beta\in \O$ independently, which gives us $\epsilon_{\beta}>0$ and $\{a_1^{\beta},...,a_{n_{\beta}}^{\beta}\}$ characterizing the spectral gap of $\beta$. It is a natural question to ask if we can obtain a uniform spectral gap, i.e. obtaining spectral gap with the same $\epsilon>0$ and the same finite set $\{a_1,...,a_n\}$ for all automorphisms in a neighborhood of $\theta$. More specifically, is it possible to establish, for some families of factors (not necessarily of type $\II_1$), whether one of the following uniform spectral gap properties is verified?
\begin{description}
    \item[P1]\label{P1} For all $\theta\in \Aut(M)\setminus \overline{\Inn}(M)$, there is a neighborhood $\V$ of $\underline{p}(\theta)$ in $\underline{\Out}(M)$ such that 
        \[\ldeux(\id)\nprec \bigoplus_{\beta\in \underline{p}^{-1}\V}\ldeux(\beta).\]
    
    Equivalently in the $\II_1$ factor case,  there exist $a_1,...,a_n\in M$, $\epsilon>0$ such that for all $\beta\in\underline{p}^{-1}(\V)$ and $\xi \in \ldeux M$,  \[ \sum_{i=1}^n \norm{\beta(a_i)\xi-\xi a_i}_2^2\geq \epsilon \norm{\xi}_2^2.\]
    
    \item[P2]\label{P2} For all neighborhoods $\V$ of $1$ in $\underline{\Out}(M)$,
    \[\ldeux(\id)\nprec \bigoplus_{\beta\in \Aut(M)\setminus \underline{p}^{-1}(\V) }\ldeux(\beta).\]
    
    Equivalently in the $\II_1$ factor case, there exist $a_1,...,a_n\in M$, $\epsilon>0$ such that for all $\beta\notin \underline{p}^{-1}(\V)$ and $\xi \in \ldeux M$,  \[ \sum_{i=1}^n \norm{\beta(a_i)\xi-\xi a_i}_2^2\geq \epsilon \norm{\xi}_2^2.\]
\end{description}
Note that \hyperref[P2]{P2} implies \hyperref[P1]{P1}. 

The relevance of these properties is supported by the following results:
\begin{The}[{\textmd{\cite[Lemma 4]{jones_centralsequences_1982}}}]\label{prop:uniform spectral gap}
 Let $M$ be a full $\II_1$ factor. Then $M$ has \hyperref[P2]{P2}.
\end{The}
Using Proposition~\ref{prop:weak equivalence with essential image} and noticing that separability assumptions can be avoided in the proof, we can rephrase \cite[Lemma 1]{marrakchivaesSpectral} as the following generalization of Jones' result: 
%there is something else hidden here: the existence of a sigma finite measure of full support on \Aut(M)or something like that.
\begin{The}[\textmd{{\cite[Lemma 1]{marrakchivaesSpectral}}}]\label{prop:uniform spectral gap marrakchivaes}
    Let $M$ be a full $\sigma$-finite factor. Then $M$ has \hyperref[P2]{P2}.
\end{The}

The proposition below can be easily adapted from \cite[Proof of Theorem A]{marrakchivaesSpectral}: 
\begin{Prop}\label{prop:proof of thm A marrakchi vaes}
 Let $\alpha:G\acts M$ be a continuous topologically outer (i.e. $\alpha^{-1}(\overline{\Inn}(M))=1$) action of a locally compact group on a factor $M$ that has \hyperref[P1]{P1}. Then $\alpha$ is strictly outer. 
\end{Prop}

It is the combination of Theorem~\ref{prop:uniform spectral gap marrakchivaes} and Proposition~\ref{prop:proof of thm A marrakchi vaes} that allows Marrakchi and Vaes to show that outer actions on full factors are strictly outer. Jones's result~\ref{prop:uniform spectral gap} strongly relies on the theorem by Connes that gives a spectral gap characterization of fullness \cite[Theorem 3.1]{connes76}, and thus its proof can't be generalized to arbitrary factors. We don't even know if the weaker property \hyperref[P1]{P1} holds for non-full factors. However, Theorem~\ref{thm:fundamental result on continuity of weak containment} asserts that such a uniform spectral gap can be obtained at least for locally compact group actions.

\section{Proof of the theorems}

First we establish a prolongation result for states on bimodules. 
\begin{Lem}\label{lem:prolongation result for states}        
 Let $M$ and $N$ be two von Neumann algebras. Let $(\H,\pi_{\H})$ be an $M$-$N$-bimodule endowed with a binormal state $\phi\in B(\H)^*$. Let $(\K,\pi_{\K})$ be another $M$-$N$-bimodule endowed with a state $\omega\in B(\K)^*$ such that:
    \begin{itemize}
        \item $\omega\circ\lambda_{\K}=\phi\circ\lambda_{\H}$. 
        \item $\omega\circ\rho_{\K}=\phi\circ\rho_{\H}$.
        \item there exist $*$-strongly dense subsets $M_0\subset M$ and $N_0\subset N$ such that $\omega\circ\pi_{\K}=\phi\circ\pi_{\H}$ on $M_0\pta N_0^{\op}$.
    \end{itemize} 
 Then $\omega\circ \pi_{\K}=\phi\circ \pi_{\H}$ on $M\pta N^{\op}$.
\end{Lem}
\begin{proof}
 Let $a\in M$, $b\in N$. Consider two  nets $(a_i)_{i\in I}\in (M_0)^{I}$ and $(b_i)_{i\in I}\in (N_0)^{I}$ such that  $a_i\tend[*\text{-strongly}]{i}{\infty}a$ and $b_i\tend[*\text{-strongly}]{i}{\infty}b$. Then:
    \begin{align*}
        &\abs{\omega(\pi_{\K}(a_i\otimes b_i^{\op}))-\omega(\pi_{\K}(a\otimes b^{\op}))}\\
            &\;=\abs{\omega(\pi_{\K}((a_i-a)\otimes b_i^{\op}))+\omega(\pi_{\K}(a\otimes (b_i-b)^{\op}))}\\
            &\; \leq \norm{a_i-a}_{\omega\circ \lambda_{\K}}\norm{(b_i^{\op})^*}_{\omega\circ \rho_{\K}}+\norm{a}_{\omega\circ \lambda_{\K}}\norm{(b_i^{\op}-b^{\op})^*}_{\omega\circ \rho_{\K}}\\
            &\; \leq \norm{a_i-a}_{\phi\circ\lambda_{\H}}^{\sharp}\norm{b_i^{\op}}_{\phi\circ\rho_{\H}}^{\sharp}+\norm{a}_{\phi\circ\lambda_{\H}}^{\sharp}\norm{(b_i-b)^{\op}}_{\phi\circ\rho_{\H}}^{\sharp}\tend{i}{\infty} 0
    \end{align*} 
 
    \begin{align*} 
            &\abs{\omega(\pi_{\K}(a_i\otimes b_i^{\op}))-\phi(\pi_{\H}(a\otimes b^{\op}))}\\
            &\;=\abs{\phi(\pi_{\H}(a_i\otimes b_i^{\op}))-\phi(\pi_{\H}(a\otimes b^{\op}))}\\
            &\;\leq \abs{\phi(\pi_{\H}((a_i-a)\otimes b_i^{\op}))+\phi(\pi_{\H}(a\otimes (b_i-b^{\op})))}\\
            &\;\leq \norm{a_i-a}_{\phi\circ\lambda_{\H}}^{\sharp}\norm{b_i^{\op}}_{\phi\circ\rho_{\H}}^{\sharp}+\norm{a}_{\phi\circ\lambda_{\H}}^{\sharp}\norm{(b_i-b)^{\op}}_{\phi\circ\rho_{\H}}^{\sharp}\tend{i}{\infty} 0. 
    \end{align*}      
\end{proof}

\begin{proof}[Proof of Theorem~\ref{thm:fundamental result on continuity of weak containment}]

Let $g\in G$, and denote by $\B$ a basis of open neighborhoods of $g$ in $G$. Assume that for all $\O\in \B$, $\ldeux(\id)\prec \bigoplus_{\theta\in \O}\ldeux(\theta)$ or equivalently that $\ldeux(\id)\prec \ldeux (\alpha_{\mid\O})$. Because $M$ is $\sigma$-finite, we can choose $\phi\in M_*^+$ a faithful normal state on $M$. Using Proposition~\ref{prop:equivalent_caracterisation_weak_containment} we choose for each $\O\in \B$ a state $\omega_{\O}\in B(\ldeux(\alpha_{\mid \O}))^*$ such that $\omega_{\O}(\pi_{\alpha_{\mid\O}}(a\otimes b^{\op}))=\ps{a\phi^{1/2}b}{\phi^{1/2}}$ for all $a,b\in M$. It induces a net $(\tilde{\omega}_{\O})_{\O\in\B}$ of states in $B(\ldeux M)^*$ defined by $\tilde{\omega}_{\O}(x)\doteq \omega_{\O}(1_{\O}\otimes x)$ for $x\in B(\ldeux M)$. Consider $\tilde{\omega}$ a weak-$*$ accumulation point of $(\tilde{\omega}_\O)_{\O\in\B}$. Let $a\in M_{\alpha}$, $b\in M$: for any $\epsilon>0$, we can find $\Omega\in\B$ small enough such that for all $h\in \Omega$, $\norm{\alpha_{g}^{-1}(a)-\alpha_h^{-1}(a)}\leq \epsilon$. Then: 
        \begin{flalign*}
            \phantom{\text{For $b\in M$ :}\phantom{aaaaa}}        &\abs{\tilde{\omega}(\pi_{\alpha_g}(a\otimes b^{\op}))-\ps{a\phi^{1/2}b}{\phi^{1/2}}}& &\\
            &\;\leq\limsup_{\O\to \infty}\abs{\omega_{\O}(1_{\O}\otimes\pi_{\alpha_g}(a\otimes b^{\op}))-\omega_{\O}(\pi_{\alpha_{\mid\O}}(a\otimes b^{\op}))}& &\\
            &\;\leq\limsup_{\O\to \infty}\abs{\omega_{\O}\Big(\big[1_{\O}\otimes\pi_{\alpha_g}(a\otimes 1)-\pi_{\alpha_{\mid\O}}(a\otimes 1)\big](\pi_{\alpha_{\mid\O}}(1\otimes b^{\op}))\Big)}& &\\
            &\;\leq \limsup_{\O\to \infty}\norm{(1_{\O}\otimes\pi_{\alpha_g}(a\otimes 1)-\pi_{\alpha_{\mid\O}}(a\otimes 1))}\norm{\pi_{\alpha_{\mid\O}}(1\otimes b^{\op})}& &\\
            &\;\leq \norm{1_{\Omega}\otimes \pi_{\alpha_g}(a\otimes 1)-\pi_{\alpha_{\mid \Omega}}(a\otimes 1)}\norm{\pi_{\alpha_{\mid \Omega}}(1\otimes b^{\op})}& &\\
            &\;\leq  \epsilon \norm{1\otimes b^{\op}}.& &
        \end{flalign*}
This being true for all $\epsilon>0$, we get that for all $a\in M_{\alpha}$ and $b\in M$, 
\begin{equation}\label{eq:nice property of intermediate state}
    \tilde{\omega}(\pi_{\alpha_g}(a\otimes b^{\op}))=\ps{a\phi^{1/2}b}{\phi^{1/2}}.
\end{equation}

Now we smooth the state $\tilde{\omega}$. Let $\C$ be a base of open neighborhoods of $1$ in $G$, and consider $\omega$ an accumulation point of $(\tilde{\omega}\circ \tilde{\alpha}_{\V})_{\V\in\C}$.
\begin{flalign*} 
    \text{For $b\in M$ :}\phantom{aaaaa}
            &\abs{\omega(\pi_{\alpha_g}(1\otimes b^{\op}))-\ps{\phi^{1/2}b}{\phi^{1/2}}}& &\\
            &\;\leq \limsup_{\V\to \infty}\abs{\tilde{\omega}\circ \tilde{\alpha}_{\V}(\pi_{\alpha_g}(1\otimes b^{\op}))-\ps{\phi^{1/2}b}{\phi^{1/2}}}& &\\
            &\;\leq\limsup_{\V\to \infty}\abs{\ps{\phi^{1/2}(\alpha_{\V}(b)-b)}{\phi^{1/2}}}& &\\
            &\;=0. 
\end{flalign*}
\begin{flalign*} 
    \text{For $a\in M$ :}\phantom{aaaaa}
            &\abs{\omega(\pi_{\alpha_g}(a\otimes 1))-\ps{a\phi^{1/2}}{\phi^{1/2}}}& &\\
            &\;\leq\limsup_{\V\to \infty}\abs{\tilde{\omega}\circ \tilde{\alpha}_{\V}(\pi_{\alpha_g}(a\otimes 1))-\ps{a\phi^{1/2}}{\phi^{1/2}}}& &\\
            &\;\leq \limsup_{\V\to \infty}\abs{\ps{(\alpha_{g\V g^{-1}}(a)-a)\phi^{1/2}}{\phi^{1/2}}}& & \\
            &\;=0. 
\end{flalign*}
\begin{flalign*}
\text{For $a,b\in M_{\alpha}$:}\phantom{aaa}
            &\abs{\omega(\pi_{\alpha_g}(a\otimes b^{\op}))-\ps{a\phi^{1/2}b}{\phi^{1/2}}}
             = \abs{\omega(\pi_{\id}(\alpha_{g^{-1}}(a)\otimes b^{\op}))-\ps{a\phi^{1/2}b}{\phi^{1/2}}} & &\\
            &\; \leq\limsup_{\V\to \infty}\abs{\tilde{\omega}\left(\frac{1}{\mu(\V)}\int_{\V}\tilde{\alpha}_h[\pi_{\id}(\alpha_{g^{-1}}(a)\otimes b^{\op})]d\mu(h)\right)-\ps{a\phi^{1/2}b}{\phi^{1/2}}}& &\\
            &\; \leq \limsup_{\V\to \infty}\abs{\frac{1}{\mu(\V)}\int_{\V}\tilde{\omega}\left(\tilde{\alpha}_h[\pi_{\id}(\alpha_{g^{-1}}(a)\otimes b^{\op})]\right)d\mu(h)-\ps{a\phi^{1/2}b}{\phi^{1/2}}}& &\tag{since $\pi_{\id}(\alpha_{g^{-1}}(a)\otimes b^{\op})\in B(\ldeux M)_{\tilde{\alpha}}$}\\ 
            &\; \leq\limsup_{\V\to \infty}\abs{\frac{1}{\mu(\V)}\int_{\V}\tilde{\omega}\Big(\pi_{\alpha_g}\Big(\alpha_{ghg^{-1}}(a)\otimes \alpha_h(b)^{\op}\Big)\Big)d\mu(h)-\ps{a\phi^{1/2}b}{\phi^{1/2}}}& &\\
            &\; \leq\limsup_{\V\to \infty}\frac{1}{\mu(\V)}\int_{\V}\abs{\tilde{\omega}\Big(\pi_{\alpha_g}\Big(\alpha_{ghg^{-1}}(a)\otimes \alpha_h(b)^{\op}\Big)\Big)-\ps{a\phi^{1/2}b}{\phi^{1/2}}}d\mu(h)& &\\ 
            &\; \leq \limsup_{\V\to \infty}\frac{1}{\mu(\V)}\int_{\V} \abs{\ps{\alpha_{ghg^{-1}}(a)\phi^{1/2}\alpha_h(b)-a\phi^{1/2}b}{\phi^{1/2}}}d\mu(h)& &\tag{by \eqref{eq:nice property of intermediate state}, since $\alpha_{ghg^{-1}}(a)\in M_{\alpha}$ }\\
            &\; \leq \limsup_{\V\to \infty}\sup_{h\in \V}\left(\Vert\alpha_{ghg^{-1}}(a)-a\Vert\norm{\alpha_h(b)}+\norm{a}\norm{\alpha_h(b)-b}\right)& &\\
            &\; =0.
\end{flalign*}
Now we can apply lemma~\ref{lem:prolongation result for states}: $\omega\circ \pi_{\alpha_g}=\phi\circ\pi_{\id}$ on $M\odot M^{\op}$ and thus $\ldeux(\alpha_g)$ weakly contains $\ldeux(\id)$ by proposition~\ref{prop:equivalent_caracterisation_weak_containment}.
\end{proof}
The following result explains how structure on the algebra of intertwiners of the action bimodule can be used to locate the relative commutant $M'\cap (M\rtimes_{\alpha}G)$ in the crossed product by a subgroup $M\rtimes_{\alpha}H$. 
\begin{Lem}\label{lem:localization of relative commutant}
 Let $\alpha:G\acts M$ be a continuous action of a locally compact group on a factor. Let $H<G$ be a closed subgroup such that $1\ptvn \linf(H\setminus G)\subset \Z(\End(\ldeux(\alpha)))$. Then 
 \[M'\cap (M\rtimes_{\alpha}G)\subset M\rtimes_{\alpha} H.\]
\end{Lem}

\begin{proof}
  
Recall that \[\End(\ldeux(\alpha))=\alpha(M)'\cap (M^{\op}\otimes 1 )'\cap (B(\ldeux M)\ptvn B(\ldeux G))=\alpha(M)'\cap (M\ptvn B(\ldeux G))\] and that the crossed product $M\rtimes_{\alpha}G$ can be realized as the algebra of fixed points in $M\ptvn B(\ldeux G)$ under the action $\alpha\otimes \Ad(\rho)$ (cf \cite[X.1.22]{takesakiTheoryII}). Using the $\alpha\otimes \Ad(\rho)$ invariance of $\alpha(M)$, we obtain that $M'\cap(M\rtimes_{\alpha}G)=\End(\ldeux(\alpha))^{\alpha\otimes \Ad(\rho)}$. Then: 
%By \cite[X.1.21]{takesaki1979}, we can describe the commutant of $M\rtimes_{\alpha}G$ in $B(\ldeux M\otimes \ldeux G)$ using the operator $W\in\linf(G)\ptvn B(\ldeux M)$, $W\xi(g)=\kappa(\alpha_g)\xi(g)$ for $\xi\in \ldeux(G)\otimes\ldeux M$ and $g\in G$. 
    \begin{flalign*}
 M'&\cap(M\rtimes_{\alpha} G) = (\End(\ldeux(\alpha))\cap\Z(\End(\ldeux(\alpha)))')^{\alpha \otimes \Ad(\rho)}& &\\
            &\subset \alpha(M)'\cap (1\ptvn \Lp^{\infty}(H\setminus G))'\cap (M\rtimes_{\alpha} G)& &\\
            &\subset    (1\ptvn\linf(H\setminus G))' \cap M\rtimes_{\alpha} G&&\\
            &= M\rtimes_{\alpha} H &&\tag{\text{by \cite[Prop. 2.2]{boutonnet2016crossedproducts}}.}
    \end{flalign*}
\end{proof}

\begin{proof}[Proof of Theorem~\ref{thm:location of relative commutant in spectrally inner automorphisms}]
 Using Theorem~\ref{thm:fundamental result on continuity of weak containment} we show that if $H=\{g\in G\mid \ldeux(\id)\prec \ldeux(\alpha_g) \}$, then \[\cc\ptvn \Lp^{\infty}(G/H)\subset \Z(\End(\ldeux(\alpha))).\] Noticing that $H$ is a normal subgroup of $G$, lemma~\ref{lem:localization of relative commutant} concludes the proof of Theorem~\ref{thm:location of relative commutant in spectrally inner automorphisms}.

Set $\N\doteq \End(\ldeux(\alpha))$. Take $K$ a compact subset of $G$. We want to establish that $1\otimes 1_{KH}\in\Z(\N)$. Because $M$ is a factor, $\N\cap (1\ptvn\linf G)'=(1\ptvn\linf G)$ and consequently $\Z(\N)\subset 1\otimes\Lp^{\infty}(G)$. Consider $A\subset G$ such that $z(1\otimes 1_{KH})=1\otimes 1_A$, where $z(\cdot)$ is the central support in $\N$. Let $L$ be the support of $\mu_{\mid A}$, and take $g_0\in L$: then $\ldeux(\alpha_{g_0})\prec\ldeux(\alpha_{L})\sim \ldeux(\alpha_{\mid A})\sim \ldeux(\alpha_{\mid KH})$ (because $\ldeux(\alpha_{\mid A})$ and $\ldeux(\alpha_{\mid KH})$ are quasi-equivalent by construction). 

Now assume that $g_0\notin KH$. Take any $g\in KH$: $g_0^{-1}g\notin H$. By Theorem~\ref{thm:fundamental result on continuity of weak containment} there exists an open neighborhood $U_g$ of $g$ such that $\ldeux(\id)\nprec\ldeux(\alpha_{\mid g_0^{-1}U_g})$. Taking the tensor product by $\ldeux(\alpha_{g_0})$ we get an open covering $(U_g)_{g\in G}$ of $KH$ such that for all $g\in G$, $\ldeux(\alpha_{g_0})\nprec \ldeux(\alpha_{\mid U_g})$. Denote by $q:G\to G/H$ the quotient map, and recall that it is continuous and open. Hence, our open covering induces an open covering $q(K)\subset\bigcup_{g\in G}q(U_g)$. We extract a finite covering $q(K)\subset \bigcup_{i=1}^{n}U_{g_i}$, and pull it back: $KH\subset\bigcup_{i=1}^n U_{g_i}H$. For $1\leq i \leq n$, $\ldeux(\alpha_{g_0})\nprec\ldeux(\alpha_{\mid U_{g_i}})\sim\ldeux(\alpha_{\mid U_{g_i}H})$ (lemma~\ref{lem:weak equivalence H(HU) et H(U)}), hence $\ldeux(\alpha_{g_0})\nprec \bigoplus_{i=1}^n \ldeux(\alpha_{\mid U_{g_i}H})\sim \ldeux(\alpha_{\mid \bigcup_{i=1}^nU_{g_i}H})$, and therefore $\ldeux(\alpha_{g_0})\nprec \ldeux(KH)$, which is a contradiction.

Thus, $g_0\in KH$, so $L\subset KH$, $\mu(A\setminus KH)=0$ and finally $1\otimes 1_{KH}\in\Z(\N)$. 
\end{proof}

\begin{proof}[Proof of Corollary~\ref{thm:top outer action are str outer}]
 Here $\alpha:G\acts M$ is a topologically outer action of a locally compact group on a type $\II$ factor $M$. Denote by $H$ the subgroup of spectrally inner automorphisms in $G$. By Theorem~\ref{the:spectral gap caracterization of topological outerness}, $H$ only contains trace-scaling automorphisms. By Theorem~\ref{thm:location of relative commutant in spectrally inner automorphisms}, $M'\cap(M\rtimes_{\alpha} G)\subset M'\cap(M\rtimes_{\alpha} H)$.  By Connes-Takesaki's relative commutant theorem \cite[XII.1.7 in combination with XII.1.1]{takesakiTheoryII}, $M'\cap(M\rtimes_{\alpha} H)=\cc$, which completes the proof.
\end{proof}

\paragraph*{Acknowledgments.} We are very grateful to our advisor Amine Marrakchi for suggesting this question and for his help during this work. We thank Cyril Houdayer, Sven Raum and Stefaan Vaes for their comments on a first draft of this article.

\printbibliography
\end{document}

%% file: preamble.tex
\usepackage[T1]{fontenc}
\usepackage[english]{babel}
\usepackage{url}
\usepackage{hyperref}
\hypersetup{colorlinks=true,linkcolor = blue,
            urlcolor  = blue,
            citecolor = blue,
            anchorcolor = blue}

\usepackage[
    url=false,
    eprint=true,
    backend=biber,
    style=numeric,
    maxnames=5]{biblatex}

\usepackage{enumitem}
\usepackage{parskip}

\usepackage{amsmath,amssymb,amsfonts,mathrsfs,amsthm}
\usepackage{mathrsfs}
\usepackage[margin=3cm]{geometry}
\usepackage{csquotes}
\usepackage{lmodern}
\usepackage{pdfsync}
\usepackage{color}

\usepackage{fancyhdr}
\usepackage{hyphenat}
\usepackage[table,xcdraw]{xcolor}
\usepackage{longtable}
\usepackage{ifthen} 
\usepackage{newunicodechar}
\usepackage{tikz-cd}
\definecolor{Myblue}{rgb}{0.2,0.2,0.7}
\usetikzlibrary{arrows}
\usepackage{adjustbox}
\usepackage{tabularx}
\usepackage{mathtools}

%% file: macros.tex
\renewcommand{\ss}{\mathbb S}

\newcommand{\nn}{\mathbb N}
\newcommand{\Lp}{L}
\newcommand{\rr}{\mathbb R}

\newcommand{\cc}{\mathbb C}

\newcommand{\B}{\mathcal B}
\newcommand{\C}{\mathcal C}

\renewcommand{\H}{\mathcal H}
\newcommand{\K}{\mathcal K}

\newcommand{\U}{\mathcal U}
\newcommand{\Z}{\mathcal Z}
\newcommand{\N}{\mathcal N}

\newcommand{\V}{\mathcal V}
\renewcommand{\L}{\mathcal L}
\renewcommand{\O}{\mathcal O}

\newcommand{\ptvn}{\overline{\otimes}}
\newcommand{\pta}{\odot}

\let\implies\Rightarrow

\let\epsilon\varepsilon

\DeclareMathOperator{\Hom}{Hom}

\DeclareMathOperator{\id}{id}

\DeclareMathOperator{\Aut}{Aut}
\DeclareMathOperator{\Inn}{Inn}
\DeclareMathOperator{\Out}{Out}

\DeclareMathOperator{\Ad}{Ad}

\DeclareMathOperator{\II}{II}
\DeclareMathOperator{\III}{III}

\DeclareMathOperator{\Bimod}{Bimod}
\DeclareMathOperator*{\esssup}{ess\,sup}
\DeclareMathOperator*{\End}{End}

\DeclareMathOperator*{\essim}{essim}

\DeclareMathOperator*{\op}{op}

\renewcommand{\phi}{\varphi}
\renewcommand{\epsilon}{\varepsilon}

\newcommand{\abs}[1]{\left \lvert #1 \right \rvert}
\newcommand{\norm}[1]{\left \Vert #1 \right \Vert}
\newcommand{\ps}[2]{\langle #1,#2 \rangle}
\newcommand{\equal}[1]{\underset{#1}{=}}

\newcommand{\tend}[3][]{\displaystyle\mathop{\overset{#1}{\longrightarrow}}_{#2\rightarrow#3}} 

\newcommand{\acts}{\curvearrowright}

\newcommand{\sjw}{false}
\newcommand{\ei}[2]
    {
        \ifthenelse{\equal{\sjw}{true}}{\kern-.25em #1}{\kern-.25em #2}
    }

\newcommand{\linf}{\Lp^{\infty}}
\newcommand{\ldeux}{\Lp^{2}}

\newtheoremstyle{withtitle}
  {3pt}
  {3pt}
  {\itshape}
  {}
  {\bfseries}
  {.}
  {.5em}
  {\thmname{#1}\thmnumber{ #2}\thmnote{ (#3)}}
  \theoremstyle{withtitle}

\newtheorem{The}{Theorem}[section]
\newtheorem{Lem}[The]{Lemma}
\newtheorem{Prop}[The]{Proposition}

\newtheorem{Ques}{Question}
\newtheorem{Prob}{Problem}

\newtheorem{mainthm}{Theorem}

\newtheorem{maincor}[mainthm]{Corollary}

\newtheorem{refthm}{Theorem}

\numberwithin{equation}{section}

\newcommand{\bimod}[3]{(#2,\pi_{#2})}
\newcommand{\MNbimod}[1]{\bimod{M}{#1}{N}}

\newcommand\lowoverline[2][a]{\ensuremath\overline{{\smash{#2}\vphantom{#1}}}}